\documentclass[a4paper,11pt]{article}

\usepackage{amsmath}
\usepackage{amssymb}
\usepackage{amsthm}
\usepackage{graphicx}
\usepackage{amscd}
\usepackage{epic, eepic}
\usepackage{url}
\usepackage{color}
\usepackage[utf8]{inputenc} 
\usepackage{comment}
\usepackage{enumerate}   
\usepackage{graphicx}
\usepackage{epstopdf}
\usepackage{enumitem}
\usepackage{tikz}
\usepackage[top=24mm, bottom=25mm, left=25mm, right=25mm]{geometry}
\usepackage[bookmarks=false,hidelinks]{hyperref}

\makeatletter
\renewenvironment{proof}[1][\proofname] {\par\pushQED{\qed}\normalfont\topsep6\p@\@plus6\p@\relax\trivlist\item[\hskip\labelsep\bfseries#1\@addpunct{.}]\ignorespaces}{\popQED\endtrivlist\@endpefalse}
\makeatother

\newtheorem{theorem}{\bf Theorem}[section]
\newtheorem{lemma}[theorem]{\bf Lemma}

\theoremstyle{definition}

\newtheorem{remark}[theorem]{\bf Remark}
\newtheorem{definition}[theorem]{\bf Definition}

\def\cG{\mathcal{G}}

\def\cC{\mathcal{C}}

\def\cP{\mathcal{P}}

\title{Regular subgraphs of linear hypergraphs}

\author{
	Oliver Janzer\thanks{Department of Mathematics, ETH, Z\"urich, Switzerland. Research supported by an ETH Z\"urich Postdoctoral Fellowship 20-1 FEL-35. Email: \textbf{oliver.janzer@math.ethz.ch}.}
	\and
	Benny Sudakov\thanks{Department of Mathematics, ETH, Z\"urich, Switzerland. Research supported in part by SNSF grant 200021\_196965. Email: \textbf{benjamin.sudakov@math.ethz.ch}.}
	\and
	Istv\'{a}n Tomon\thanks{Department of Mathematics and Mathematical Statistics, Ume\r{a} University, Sweden. Email: \textbf{istvan.tomon@umu.se}}
}

\date{}

\begin{document}

\maketitle

\sloppy

\begin{abstract}
    We prove that the maximum number of edges in a 3-uniform \emph{linear} hypergraph on $n$ vertices containing no 2-regular subhypergraph is $n^{1+o(1)}$. This resolves a conjecture of Dellamonica, Haxell, \L uczak, Mubayi, Nagle, Person, R\"odl, Schacht and Verstra\"ete. We use this result to show that the maximum number of edges in a $3$-uniform hypergraph on $n$ vertices containing no immersion of a closed surface is $n^{2+o(1)}$. Furthermore, we present results on the maximum number of edges in $k$-uniform linear hypergraphs containing no $r$-regular subhypergraph.
\end{abstract}
	
\section{Introduction}

A non-empty hypergraph is called $r$-regular if every vertex in it has degree $r$. How many edges can an $n$-vertex graph, or more generally  $k$-uniform hypergraph (or $k$-graph, for short), have without containing an $r$-regular subhypergraph? This fundamental question has attracted considerable interest in the last few decades. The simplest instance of the problem, concerning the $r=2$ case for graphs, is equivalent to determining the maximum number of edges in an $n$-vertex acyclic graph. Clearly, the answer for this question is $n-1$, the extremal example being any $n$-vertex tree. The problem immediately becomes much more difficult if we increase $k$ or $r$. The case of general $r$ for graphs was a problem raised by Erd\H os and Sauer in the 70's \cite{Erd75}. After decades of partial results, it was very recently showed by Janzer and Sudakov \cite{JS22} that the answer to this question is $\Theta_r(n\log \log n)$.

Concerning hypergraphs, Mubayi and Verstra\"ete \cite{MV09} proved that for any even $k$ and sufficiently large $n$, the maximum number of edges in an $n$-vertex $k$-graph which does not contain a $2$-regular subhypergraph is $\binom{n-1}{k-1}$, the extremal example being a star. Complementing this result, Han and Kim \cite{HK18} showed that if $k$ is odd and $n$ is sufficiently large, then the maximum number of edges in an $n$-vertex $k$-graph without a $2$-regular subhypergraph is $\binom{n-1}{k-1}+\lfloor \frac{n-1}{k} \rfloor$, and the extremal example is a star with centre $v$ together with a full matching avoiding $v$. Kim \cite{Kim16} obtained exact and asymptotic results for certain further pairs $(r,k)$ as well.

A hypergraph is \emph{linear} if any two edges intersect in at most one vertex. When extending extremal results about graphs to hypergraphs, it is usually natural to consider linear hypergraphs as they are more similar to graphs than general hypergraphs are. The above problem for linear hypergraphs was studied by Dellamonica, Haxell, \L uczak, Mubayi, Nagle, Person, R\"odl, Schacht and Verstra\"ete \cite{DHL+}. They were particularly interested in $2$-regular subhypergraphs, and more generally subhypergraphs in which all vertices have bounded even degree. The additional condition of linearity seems to make the problem more challenging, as there is no good candidate for an extremal example. Let us write $f(n)$ for the maximum possible number of edges in an $n$-vertex, linear $3$-graph which does not have a $2$-regular subhypergraph. Dellamonica et al. \cite{DHL+} proved that 
$$\Omega(n\log n)\leq f(n)\leq O(n^{3/2}(\log n)^5).$$
They conjectured that their lower bound is asymptotically correct, that is, $f(n)=n^{1+o(1)}.$ Supporting this conjecture, they proved that there is a constant $C$ such that any $n$-vertex linear 3-graph with at least $Cn(\log n)^2$ edges contains a subhypergraph in which all degrees are even and at most $12$. Our first result proves their conjecture.

\begin{theorem} \label{thm:2-reg}
Every linear $3$-graph with $n$ vertices and at least $n^{1+o(1)}$ edges 
contains a $2$-regular subhypergraph. More precisely, there is a constant $C$ such that  $$f(n)\leq n\exp(C\sqrt{(\log n)(\log \log n)}).$$
\end{theorem}

Dellamonica et al.  \cite{DHL+} also studied the problem of finding \emph{small} even-degree subhypergraphs. This question also comes up naturally in coding theory, where even-degree subhypergraphs correspond to linearly dependent vectors, see \cite{NV}. For a positive integer $t$, let us write $f(n;t)$ for the maximum number of edges in an $n$-vertex, linear $3$-graph which does not contain a $2$-regular subhypergraph with at most $t$ edges. They conjectured that there is a constant $C$ such that for every fixed positive integer $t$, we have $f(n;t)=O(n^{3/2+C/t})$. They noted that a standard probabilistic argument (see \cite{NV}) shows that this would be tight up to the value of $C$. We prove this conjecture as well in the following form.

\begin{theorem} \label{thm:bounded 2-reg}
    For any integer $\ell\geq 2$, there exists a constant $c=c(\ell)$ such that
    $$f(n;4\ell)\leq c n^{3/2+1/\ell}.$$
\end{theorem}

Another motivation for studying 2-regular subhypergraphs comes from a celebrated conjecture of Feige \cite{F08}. An \emph{even hypergraph} is a non-empty hypergraph in which every vertex is contained in an even number of edges. An elementary linear algebraic argument shows that every hypergraph with $n$ vertices and at least $n+1$ edges contains an even subhypergraph. Motivated by average-case problem of refuting random 3SAT formulas, Feige conjectured that for every $k$ and $\ell$, every $k$-uniform hypergraph on $n$ vertices and $\Omega_{k}(n(\frac{n}{\ell})^{k/2-1})$ edges contains an even subhypergraph on at most $O_{k}(\ell \log n)$ vertices. Up to polylogarithmic factors, this conjecture was recently resolved by Guruswami, Kothari, and Manohar \cite{GKM22}, noting that $k=3$ is among the more challenging cases. For linear 3-graphs, one can view  Theorems \ref{thm:2-reg} and \ref{thm:bounded 2-reg} as a strengthening of this in the extreme cases $\ell=n$ and $\ell=\Theta(1)$ (with worse bounds on the required number of edges but with a better bound on the size of the subhypergraph), where we guarantee a 2-regular subhypergraph.

\medskip

We  conclude by discussing the case of general $k$ and $r$. Let us write $f_r^{(k)}(n)$ for the maximum possible number of edges in an $n$-vertex linear $k$-graph which does not contain an $r$-regular subhypergraph. Dellamonica et al. \cite{DHL+} observed that $f_2^{(k)}(n)\leq Cn^{2-\frac{1}{k-1}}(\log n)^{Ck}$ for some absolute constant $C$. They asked whether any large linear $3$-graph with a quadratic number of edges must contain a $3$-regular subhypergraph, that is, whether $f_3^{(3)}(n)=o(n^2)$. This was answered affirmatively by Kim \cite{Kim16}, who proved that for any pair of integers $k,r\geq 3$ and sufficiently large $n$, we have
   $f_r^{(k)}(n)<6n^2(\log \log n)^{-\frac{1}{2(k-1)}}.$
We greatly improve this upper bound, and provide a lower bound which matches our upper bound up to logarithmic factors if $r=\Omega(\log n)$.

\begin{theorem}\label{thm:k-uni_r-reg}
    For any $k\geq 2$, there are positive constants $C_0=C_0(k)$, $C_1=C_1(k)$ and $C_2=C_2(k)$ such that for every $2\leq r\leq n$, we have
    $$C_0 r^{\frac{1}{k-1}} n^{2-\frac{r}{(r-1)(k-1)}} \leq f_r^{(k)}(n)\leq C_1r^{\frac{1}{k-1}}n^{2-\frac{1}{k-1}}(\log n)^{C_2}.$$
\end{theorem}

\subsection{Hypergraph immersion}

To conclude the introduction, let us discuss an interesting topological corollary of Theorem \ref{thm:2-reg}. A $k$-graph $G$ naturally corresponds to the $(k-1)$-dimensional \emph{pure simplicial complex} (i.e., a simplicial complex in which every simplex of dimension less than $k-1$ is a face of a simplex of dimension exactly $k-1$) that is the downset generated by the edges of $G$. Therefore, we can talk about $k$-graphs and $(k-1)$-dimensional simplicial complexes interchangeably. Say that a $k$-graph (or topological space) $H$ \emph{embeds} into the $k$-graph $G$ if $G$ contains a \emph{homeomorphic} copy of $H$. E.g. the complete graph $K_t$ embeds into a graph $G$ if and only if $G$ contains a subdivision of $K_t$. An old result of Brown, Erd\H{o}s, and S\'os \cite{BES73} states that if the sphere $S^2$ does not embed into a 3-graph $G$ on $n$ vertices, then $e(G)=O(n^{5/2})$, and this bound is the best possible. This was extended by Kupavskii, Polyanskii, Tomon, and Zakharov \cite{KPTZ20} who showed that if $S$ is an orientable closed surface and $S$ does not embed into $G$, then $e(G)=O_{S}(n^{5/2})$. Furthermore, they showed that there exists a $3$-graph on $n$ vertices  with $\Omega(n^{5/2})$ edges which contains no embedding of \emph{any} closed surface\footnote{By the Classification theorem of closed surfaces, a closed surface is either a sphere, the connected sum of a finite number of tori, or the connected sum of a finite number of projective planes.}. However, if we relax the notion of embedding to allow the merging of vertices, such a construction is no longer possible. Let us explain this in more detail.

In the case of graphs, Nash-Williams \cite{N65} introduced the notion of \emph{immersion} as a relaxation of subdivision. A graph $G$ contains a graph $H$ as an \emph{immersion} if there exist a function $f:V(H)\rightarrow V(G)$ and for every $uv\in E(H)$ a path $P_{uv}$ in $G$ with endpoints $f(u)$ and $f(v)$ such that the paths $\{P_{uv}\}_{uv\in E(H)}$ are edge-disjoint. In other words, there is a homomorphism $\phi:H'\rightarrow G$ such that $H'$ is homeomorphic to $H$ and $\phi$ is injective on the edges. Here, given $(k-1)$-dimensional simplicial complexes $H$ and $G$, a function $\phi:H\rightarrow G$ is a \emph{homomorphism} (not to be confused with homeomorphism) if $\phi$ maps $i$-faces to $i$-faces for every $i=0,\dots,k-1$, and if $\{x_0,\dots,x_i\}$ is an $i$-face of $H$, then $\phi(\{x_0\dots,x_i\})=\{\phi(x_0),\dots,\phi(x_i)\}$.  The extremal number of immersions of $K_t$ are extensively studied \cite{DDFMMS,DY18,GLW19}, motivated by an old conjecture of Lescure and Meyniel \cite{LM88}.

  One can extend the definition of immersion to $k$-graphs in the following natural way. Say that a $k$-graph $G$ contains an \emph{$\ell$-immersion} of $H$ if there exists a homomorphism $\phi:H'\rightarrow G$ such that $H'$ is a homeomorphic copy of $H$, and $\phi$ is injective on the $i$-faces for every $\ell<i\leq k-1$. In other words $H$ can be almost embedded in $G$, but some of the lower dimensional faces might get merged. Here is the promised corollary of Theorem \ref{thm:2-reg}.

\begin{theorem}\label{thm:immersion}
	Let $G$ be a 3-graph on $n$ vertices which contains no 0-immersion of any closed surface. Then 
	$$e(G)\leq n^{2+o(1)}.$$
\end{theorem}

The bound in Theorem \ref{thm:immersion} is also best possible up to the $o(1)$ error term. Indeed, there exist linear 3-graphs on $n$ vertices with $\Omega(n^2)$ edges (e.g. Steiner triple systems), and such a 3-graph clearly cannot contain a 0-immersion of a surface.

\medskip

\noindent
\textbf{Organization.} Our paper is organized as follows. In the next section we provide some preliminary results. Then, in Section \ref{sect:2-reg}, we prove Theorem \ref{thm:2-reg}, in Section \ref{sect:2-reg_bounded}, we prove Theorem \ref{thm:bounded 2-reg}, in Section \ref{sect:r-reg}, we prove Theorem \ref{thm:k-uni_r-reg} and in Section \ref{sect:immersion} we prove Theorem \ref{thm:immersion}. We finish our paper with some concluding remarks.

\section{Preliminaries}

In this section, we present some basic results that will be used throughout the paper.

We will use the following special case of the well known Chernoff bound. The inequalities follow fairly easily from Theorem 4 in \cite{chernoff}.

\begin{lemma}[Multiplicative Chernoff bound]\label{lemma:chernoff}
 Let $X$ be the sum of independent indicator random variables. If $\lambda\geq 2\mathbb{E}(X)$, then
 $$\mathbb{P}(X\geq \lambda)\leq e^{-\lambda/6}.$$
 Also,
 $$\mathbb{P}(X\leq \mathbb{E}(X)/2)\leq e^{-\mathbb{E}(X)/8}.$$
\end{lemma}

Following \cite{DHL+}, we say that a $k$-partite $k$-uniform hypergraph $G$ with parts $X_1,\dots,X_k$ is \emph{$\mu$-balanced} if for each $1\leq i\leq k$,
\begin{eqnarray}
\label{def1}
\max_{w\in X_i} d_G(w)\leq \frac{\mu e(G)}{|X_i|}.
\end{eqnarray}

The following lemma is a straightforward generalization of Lemma 6 from \cite{DHL+} (which is essentially the same result for $k=3$).

\begin{lemma} \label{lem:DHL regularization}
    Let $G$ be an $n$-vertex linear $k$-graph and let $\lambda=\lceil \log n\rceil$. Then $G$ contains a $2\lambda^k$-balanced $k$-partite subhypergraph with at least $e(G)k!/(k\lambda)^k$ edges.
\end{lemma}

Although the proof of this lemma is almost identical to that of Lemma 6 in \cite{DHL+}, we include it for completeness.

\begin{proof}
    By a considering random partition of $V(G)$ into $k$ parts, there is a partition $V(G)=X_1\cup\dots\cup X_k$ such that the $k$-partite subhypergraph $G'$ of $G$ with parts $X_1,\dots,X_k$ has at least $k! e(G)/k^{k}$ edges. By discarding isolated vertices, we may assume that every vertex in $G'$ has degree at least one. For each $1\leq i\leq k$, define a partition $X_i=X_i^1\cup \dots \cup X_i^{\lambda}$ by setting $X_i^j=\{x\in X_i: 2^{j-1}\leq d_{G'}(x)<2^j\}$. Note that there exist $1\leq j_1,\dots,j_k\leq \lambda$ such that $H=G'[X_1^{j_1}\cup \dots \cup X_k^{j_k}]$ has at least $e(G')/\lambda^k$ edges. On the other hand, clearly, $|X_i^{j_i}|2^{j_i-1}\leq e(G')$, so $$\max_{w\in X_i^{j_i}} d_H(w)\leq 2^{j_i}\leq \frac{2e(G')}{|X_i^{j_i}|}\leq \frac{2\lambda^ke(H)}{|X_i^{j_i}|},$$ as desired.
\end{proof}

The next result is a simple extension of the $k=3$ case of Lemma \ref{lem:DHL regularization} which shows that we can in addition assume that the two largest parts have equal size.

\begin{lemma} \label{lem:balanced regularization}
    Let $n$ be sufficiently large and let $d\geq 1$. Let $G$ be an $n$-vertex  linear $3$-graph with at least $nd$ edges and let $\lambda=\lceil\log n \rceil$. Then $G$ contains a $96\lambda^6$-balanced subhypergraph $H$ with at least $|V(H)|d/(81\lambda^3)$ edges and parts $X,Y,Z$ such that $|X|\leq |Y|=|Z|$.
\end{lemma}

\begin{proof}
    By Lemma \ref{lem:DHL regularization}, $G$ has a $2\lambda^3$-balanced $3$-partite subhypergraph $F$ with at least $2e(G)/(9\lambda^3)$ edges. Let $F$ have parts $A$, $B$ and $C$. Starting with $F'=F$, $A'=A$, $B'=B$ and $C'=C$, repeat the following procedure. For $W\in \{A,B,C\}$, if there exists $w\in W'$ such that $d_{F'}(w)<\frac{e(F)}{6|W|}$, remove $w$ from $F'$ and $W'$; otherwise stop. In total, we removed at most $|A|\cdot \frac{e(F)}{6|A|}+|B|\cdot \frac{e(F)}{6|B|}+|C|\cdot \frac{e(F)}{6|C|}=\frac{e(F)}{2}$ edges, so at the end of the process, $e(F')\geq \frac{e(F)}{2}$. Also, for every $W\in \{A,B,C\}$, $\min_{w\in W'} d_{F'}(w)\geq \frac{e(F)}{6|W|}$. Since $$\frac{e(F)}{2}\leq e(F')\leq |W'|\max_{w\in W'} d_{F'}(w)\leq |W'|\max_{w\in W} d_{F}(w)\leq  \frac{|W'|2\lambda^3 e(F)}{|W|},$$ we have $\frac{|W'|}{|W|}\geq \frac{1}{4\lambda^3}$.
    
    Assume, without loss of generality, that (at the end of the above process) $|A'|\leq |B'|\leq |C'|$. Let $X=A'$, let $Y=B'$, and let $Z$ be a uniformly random subset of $C'$ of size $|B'|$. Let $H$ be the induced subhypergraph of $F'$ with parts $X$, $Y$ and $Z$. Note that $$e(H)\geq |Z|\min_{w\in C'} d_{F'}(w)\geq \frac{|Z|e(F)}{6|C|}\geq \frac{|Z|e(G)}{27\lambda^3 |C|}\geq \frac{|Z|e(G)}{27\lambda^3 n}\geq \frac{|Z|d}{27\lambda^3}.$$ Since $|V(H)|\leq |X|+|Y|+|Z|\leq 3|Z|$, it follows that $H$ has at least $|V(H)|d/(81\lambda^3)$ edges.
    
    It remains to verify that $H$ is $96\lambda^6$-balanced. First, note that $$\max_{z\in Z} d_H(z)\leq \max_{z\in C} d_F(z)\leq \frac{2\lambda^3 e(F)}{|C|}.$$ On the other hand, $e(H)\geq \frac{|Z|e(F)}{6 |C|}$, so $\max_{z\in Z} d_H(z)\leq \frac{12\lambda^3 e(H)}{|Z|}$.
    
    Now let $W\in \{X,Y\}$ and let $w\in W$. Note that $d_{F'}(w)\leq d_F(w)\leq \frac{2\lambda^3 e(F)}{|W|}$, so the expected value of $d_H(w)$ is at most $\frac{|Z|}{|C'|}\cdot \frac{2\lambda^3 e(F)}{|W|}$. Using $\frac{|C'|}{|C|}\geq \frac{1}{4\lambda^3}$ and $e(H)\geq \frac{|Z|e(F)}{6|C|}$, we can conclude that $$\mathbb{E}(d_H(w))\leq \frac{48\lambda^6 e(H)}{|W|}.$$ Note that $e(H)\geq \frac{|Z|d}{27\lambda^3}\geq \frac{|W|d}{27\lambda^3}$, so $\frac{\lambda^6 e(H)}{|W|}\geq \frac{1}{27}\lambda^3 d\geq \frac{1}{27}(\log n)^3$.
    
    Since $G$ is linear, $d_{H}(w)$ is the sum of independent indicator random variables, so by Lemma~\ref{lemma:chernoff}, we have  
    $$\mathbb{P}\left(d_H(w)\geq \frac{96\lambda^6 e(H)}{|W|}\right)\leq e^{-\frac{16\lambda^6 e(H)}{|W|}}< e^{-2\log n}=\frac{1}{n^2}.$$
    Therefore, with probability greater than $1/2$, for every $w\in X$, we have $d_H(w)\leq \frac{96\lambda^6 e(H)}{|X|}$. It follows by symmetry that with probability greater than $1/2$, we have $d_H(w)\leq \frac{96\lambda^6 e(H)}{|Y|}$ for every $w\in Y$. Thus, with positive probability $H$ is $96\lambda^6$-balanced.
\end{proof}

Finally, with slight abuse of notation, we define the notion $\mu$-balanced for edge coloured graphs as well.

\begin{definition}
Say that an edge coloured graph $G$ on $n$ vertices with average degree $d$ is \emph{$\mu$-balanced} if $\Delta(G)\leq \mu d$, there are $s\leq n$ colours appearing on the edges of $G$, and each colour is used at most $nd\mu/s$ times. 
\end{definition}

\section{$2$-regular subhypergraphs}\label{sect:2-reg}

In this section, we prove Theorem \ref{thm:2-reg}. As this is the most involved part of our paper, let us give a brief outline.

A \emph{homomorphic copy} of a graph $H$ in a graph $G$ is a function $\phi: V(H)\rightarrow V(G)$ that maps edges to edges. Let $\hom(H,G)$ denote the number of homomorphic copies of $H$ in $G$. A \emph{labelled copy} of $H$ is an injective homomorphic copy.

Let $H$ be a linear 3-graph on $n$ vertices with average degree $d\geq \exp(C\sqrt{(\log n)(\log\log n)})$ for some large constant $C$. By Lemma \ref{lem:balanced regularization}, after possibly passing to a subhypergraph, we may assume that $H$ is $3$-partite and $\mu$-balanced with $\mu=(\log n)^{O(1)}$, and the vertex classes $X,Y,Z$ satisfy $|X|\leq |Y|=|Z|$. 

One can view $H$ as a properly edge coloured graph $G$. Indeed, let $G$ be the graph on vertex set $Y\cup Z$ such that for every $xyz\in E(H)$,  $yz$ is an edge of $G$ of colour $x$. Then $G$ is $\mu$-balanced. Our goal is to find a 2-regular subgraph of $G$ in which each colour is used exactly twice or zero times. We proceed different ways depending on the number homomorphic copies of the cycle $C_{2k}$, where $k$ depends on the number of colours $s=|X|$.

First, we consider the case where the number of homomorphic $2k$-cycles is much larger than the number of paths of length $k$, this can be found in Section \ref{sect:many_cycles}. In this case, we count collections $\cC$ of vertex-disjoint copies of $C_{2k}$ such that no colour is used more than once. If one can find two such collections $\cC$ and $\cC'$ that have the same set of colours, then the union of $\cC$ and $\cC'$ satisfies that every colour is used twice. However, some of the cycles in $\cC\cup \cC'$ might share a vertex, so it does not necessarily give a 2-regular subgraph. In order to avoid this, we add the additional condition that no cycle in $\cC$ can touch a colour that is used in another cycle. This (almost) ensures that if two cycles in $\cC$ and $\cC'$ share a vertex, then they coincide, so we can remove both of them, and we get the desired 2-regular subgraph in the end.

Second, we consider the case where the number of homomorphic $2k$-cycles is not much larger than the number of paths of length $k$, this can be found in Section \ref{sect:few_cycles}. The average degree condition gives an easy lower bound for the number of rainbow paths of length $k$. Since the number of homomorphic copies of $C_{2k}$ is not too large compared to the number of paths of length $k$, we can select a constant proportion of these rainbow paths in a way that between any two vertices there is at most one path selected. Let us call the set of selected paths $\mathcal{S}$. Now we count collections $\cP$ of vertex-disjoint paths in $\mathcal{S}$ such that no colour is used more than once. If one can find two such collections $\cP$ and $\cP'$ that have the same set of colours and the same set of endpoints, then the union of $\cP$ and $\cP'$ satisfies that every colour is used twice. However, some of the paths in $\cP\cup \cP'$ might share a vertex other than an endpoint, so $\cP\cup \cP'$ does not necessarily give a 2-regular subgraph. In order to fix this, we add the extra condition that any two paths in $\cP$ must have distance greater than $k-1$.  This ensures that if two paths in $\cP\cup \cP'$ share an internal vertex, then they have the same endpoints, so (as they both belong to $\mathcal{S}$), they coincide. Hence, we can remove both of them, and we get the desired 2-regular subgraph in the end.

\subsection{Many homomorphic cycles}\label{sect:many_cycles}

 We make use of the following lemma, which is a straightforward consequence of Lemmas 2.1 and 2.2 from \cite{Jan20}.

\begin{lemma}[\cite{Jan20}]\label{lemma:hom}
Let $k\geq 2$ be an integer and let $G$ be a properly edge-coloured  graph on $n$ vertices. Then the number of homomorphic copies of the $2k$-cycle which are not rainbow labelled copies of $C_{2k}$ is at most
$$64k^{3/2}\Delta(G)^{1/2}n^{\frac{1}{2k}}\hom(C_{2k},G)^{1-\frac{1}{2k}}.$$
\end{lemma}

\begin{definition} \label{def:nice cycles}
    In a graph $G$, we call a sequence of cycles $\cC_1,\dots,\cC_t$ \emph{nice} if for every $i\neq j$, no vertex in $\cC_i$ is incident to any edge in $G$ whose colour appears on $\cC_j$.
\end{definition}

\begin{remark}
    Note that in particular this means that $\cC_1,\dots,\cC_t$ are pairwise vertex-disjoint and no colour appears on more than one of the cycles.
\end{remark}

\begin{lemma} \label{lem:find cycles}
	Let $G$ be a properly edge-coloured $\mu$-balanced  graph  on $n$ vertices of average degree $d$ and let
	$s$ be the number of colours in this edge-colouring. Let $k\geq 2$ be an integer and assume that in every subgraph $G'$ of $G$ with at least $e(G)/2$ edges, we have $\hom(C_{2k},G')\geq (128k^{3/2}\mu^{1/2})^{2k}nd^{k}$. Then, for $t=\lceil\frac{s}{32d\mu^2 k}\rceil$, $G$ has at least $(nd^{k})^t$ nice sequences of labelled rainbow $2k$-cycles $\cC_1,\dots,\cC_t$.
\end{lemma}

\begin{proof}
	We prove the lemma by showing that once we have found suitable cycles $\cC_1,\dots,\cC_\ell$ (for some $\ell<t$), the number of suitable choices for $\cC_{\ell+1}$ is at least $nd^{k}$. Let $G'$ be the subgraph obtained from $G$ by removing
	\begin{enumerate}
		\item all vertices that are incident to any edge whose colour appears on $\cC_i$ for some $i\leq \ell$ (in particular we remove all vertices appearing in some $\cC_i$)
		\item and all edges whose colour is incident to a vertex in $\cC_i$ for some $i\leq \ell$.
	\end{enumerate}
	Observe that any rainbow $2k$-cycle in $G'$ is a suitable choice for $\cC_{\ell+1}$.
	
	Let us estimate how many edges we have removed from $G$ to get $G'$. The first kind of removal deletes at most $2nd\mu/s$ vertices for each colour that appears on $\cC_1,\dots,\cC_\ell$; this gives at most $$2k\ell\cdot \frac{2nd\mu}{s}\leq 2k(t-1)\cdot \frac{2nd\mu}{s}\leq \frac{n}{8\mu}$$ vertices in total. Each vertex is incident to at most $d\mu$ edges, so the first kind of removal deletes at most $nd/8= e(G)/4$ edges from $G$.
	
	There are $2k\ell$ vertices appearing on one of $\cC_1,\dots,\cC_\ell$, and in total these vertices are incident to at most $2k\ell\cdot d\mu $ edges. Since any colour appears at most $nd\mu/s$ times in the graph, the second kind of removal deletes at most $$2k\ell\cdot d\mu\cdot \frac{nd\mu}{s}\leq \frac{2k(t-1)nd^2\mu^2}{s}< \frac{nd}{8}= \frac{e(G)}{4}$$ edges.
	
	Hence, $e(G')\geq e(G)-2e(G)/4= e(G)/2$.  
	By Lemma \ref{lemma:hom}, the number of homomorphic copies of the $2k$-cycle in $G'$ which are not rainbow labelled copies is at most $$64k^{3/2}\Delta(G)^{1/2}n^{\frac{1}{2k}}\hom(C_{2k},G')^{1-\frac{1}{2k}}.$$ If the inequality $\hom(C_{2k},G')\geq (128k^{3/2})^{2k}n\Delta(G)^k$ is satisfied, then the previous expression is at most $\frac{1}{2}\hom(C_{2k},G')$, implying that there are at least $\frac{1}{2}\hom(C_{2k},G')$ labelled rainbow $2k$-cycles in $G'$. But noting that $\Delta(G)\leq d\mu$ and $e(G')\geq e(G)/2$, it is indeed satisfied by the conditions of the lemma. Thus, $G'$ has at least $\frac{1}{2}\hom(C_{2k},G')\geq\frac{1}{2}(128k^{3/2}\mu^{1/2})^{2k}nd^{k}> nd^{k}$ labelled rainbow $2k$-cycles. This completes the proof.
\end{proof}

\begin{lemma} \label{lem:two cycles}
	Let $G$ be a properly edge-coloured graph. Let $A_1\cup \dots \cup A_{2k}$ be a partition of the vertex set and let $B_1\cup \dots \cup B_{2k}$ be a partition of the colour set. Let $\cC_1,\dots,\cC_t$ be a nice sequence of labelled rainbow $2k$-cycles such that for any $i,\ell$, the $\ell$th vertex of $\cC_i$ is in $A_\ell$ and the colour of the $\ell$th edge of $\cC_i$ is in $B_\ell$. Let $\cC'_1,\dots,\cC'_t$ be another sequence with the analogous property. Assume, finally, that the two sequences of cycles use the same set of colours. Then if some vertex belongs to both $\cC_i$ and $\cC'_j$, then $\cC_i=\cC'_j$.
\end{lemma}

\begin{proof}
	Let $u$ be a vertex which belongs to both $\cC_i$ and $\cC'_j$. Without loss of generality, assume that $u\in A_1$. We prove by induction that for all $1\leq \ell\leq 2k$, the $\ell$th vertex of $\cC_i$ is the same as the $\ell$th vertex of $\cC'_j$. The case $\ell=1$ is given. Let $v_p$ and $w_p$ be the $p$th vertex of $\cC_i$ and $\cC'_j$, respectively. Assume that we already know that $v_\ell=w_\ell$, for some $1\leq \ell\leq 2k-1$. Let us call the colour of the edge $v_\ell v_{\ell+1}$ red. By assumption, the colour red also appears on some edge in some $\cC'_q$. But then $q=j$, otherwise the colour red is both incident to a vertex of $\cC'_{j}$, and is contained in $\cC'_q$, contradicting the definition of nice. Therefore, the colour red appears on $\cC'_j$. However, since $v_{\ell} v_{\ell+1}$ is red, the colour red belongs to $B_\ell$. Hence, the edge $w_{\ell} w_{\ell+1}$ must be the red one in $\cC'_j$. Thus, $w_{\ell+1}=v_{\ell+1}$ as $G$ is properly coloured. This completes the induction step.
\end{proof}

\begin{lemma} \label{lem:2-reg with cycles}
	Let $G$ be a properly edge-coloured $\mu$-balanced graph with at most $n$ vertices and average degree $d\geq 2$, and let
	$s$ be the number of colours in this edge-colouring. Let $k\geq 2$ be an integer for which $s\leq nd^{-k+1}(100\mu k)^{-4k}$.  Assume that in every subgraph $G'$ of $G$ with at least $e(G)/2$ edges, we have $\hom(C_{2k},G')\geq (128k^{3/2}\mu^{1/2})^{2k}nd^{k}$. Then, $G$ has a 2-regular subgraph in which every colour appears exactly twice or zero times.
\end{lemma}

\begin{proof}
	Let $t=\lceil\frac{s}{32d\mu^2 k}\rceil$. By Lemma \ref{lem:find cycles}, there are at least $(nd^{k})^t$ nice sequences of labelled rainbow $2k$-cycles $\cC_1,\dots,\cC_t$. Let $A_1\cup \dots \cup A_{2k}$ be a uniformly random partition of the vertex set and let $B_1\cup \dots \cup B_{2k}$ be a uniformly random partition of the colour set. The probability that a given nice sequence $\cC_1,\dots,\cC_t$ of rainbow $2k$-cycles satisfies that for any $i,j$, the $j$th vertex of $\cC_i$ is in $A_j$ and the colour of the $j$th edge of $\cC_i$ is in $B_j$ is precisely $(2k)^{-4kt}$. Hence, for some choices of $A_1\cup \dots \cup A_{2k}$ and $B_1\cup \dots \cup B_{2k}$, there are at least $(2k)^{-4kt}(nd^{k})^t$ nice sequences of labelled rainbow $2k$-cycles $\cC_1,\dots,\cC_t$ such that for any $i,j$, the $j$th vertex of $\cC_i$ is in $A_j$ and the colour of the $j$th edge of $\cC_i$ is in $B_j$.
	
	We claim that $(2k)^{-4kt}(nd^{k})^{t}>t!\cdot \binom{s}{2kt}$. Indeed, $t!\cdot \binom{s}{2kt}\leq t^t\cdot (\frac{es}{2kt})^{2kt}$, so it suffices to prove that $(2k)^{-4k}nd^{k}>t(\frac{es}{2kt})^{2k}$. But using $t\geq \frac{s}{32d\mu^2 k}$ and $s\leq nd^{-k+1}(100\mu k)^{-4k}$, we obtain that $$t\left(\frac{es}{2kt}\right)^{2k}\leq sd^{-1}(16ed\mu^2)^{2k}\leq nd^{-k}(100\mu k)^{-4k} (16ed\mu^2)^{2k} < (2k)^{-4k}nd^k.$$
	Hence, indeed $(2k)^{-4kt}(nd^{k})^t>t!\cdot \binom{s}{2kt}$.
	
	It follows that there is some set of $2kt$ colours $S$ for which there are more than $t!$ nice sequences of labelled rainbow $2k$-cycles $\cC_1,\dots,\cC_t$ whose colour set is precisely $S$ and such that for any $i,j$, the $j$th vertex of $\cC_i$ is in $A_j$ and the colour of the $j$th edge of $\cC_i$ is in $B_j$. Then clearly there are two such collections which are not the same set of cycles up to reordering. By Lemma~\ref{lem:two cycles}, their symmetric difference is a (non-empty) $2$-regular graph in which every colour appears exactly twice or zero times.
\end{proof}

\subsection{Few homomorphic cycles}\label{sect:few_cycles}

\begin{lemma}\label{lemma:rainbow_paths}
Let $G$ be a properly edge coloured graph with $n$ vertices and at least $nd$ edges, and let $1\leq k<d/8$. Then $G$ contains at least $n(d/4)^{k}$ labelled  rainbow paths of length $k$.
\end{lemma}

\begin{proof}
In case $k=1$, the statement is trivial, so assume $k\geq 2$. We proceed by induction on $n$. If $n=1$, there is nothing to prove, so let us assume that $n\geq 2$. First, suppose that $G$ has a vertex $v$ of degree less than $d/2$. Let $G'$ be the graph we get after removing $v$ from $G$. Then $G'$ has at least $nd-d/2$ edges, so applying our induction hypothesis with $d'=\frac{nd-d/2}{n-1}$ instead of $d$, we get that $G'$ contains at least
\begin{align*}(n-1)\left(\frac{nd-d/2}{4(n-1)}\right)^{k}&=(n-1)\left(\frac{d}{4}\right)^{k}\left(1+\frac{1}{2(n-1)}\right)^{k}\\
&\geq (n-1)\left(\frac{d}{4}\right)^{k}\left(1+\frac{k}{2(n-1)}\right)\geq n\left(\frac{d}{4}\right)^{k}
\end{align*}
labelled rainbow paths of length $k$. Hence, in this case we are done. 

Therefore, we can assume that the minimum degree of $G$ is at least $d/2$. Starting from any vertex $v_0\in V(G)$, we can build a rainbow path with vertices $v_0,\dots,v_k$ using the following simple procedure. If $v_0,\dots,v_i$ are already selected, then $v_{i+1}$ is any neighbour of $v_{i}$ that is not in $\{v_0,\dots,v_{i}\}$, and such that $v_iv_{i+1}$ is of different colour than the previously selected edges. As $2k<d/4$ and the minimum degree of $G$ is at least $d/2$, at each step we have at least $d/4$ choices for $v_{i+1}$, giving at least $n(d/4)^{k}$ labelled rainbow paths. 
\end{proof}

\begin{definition}
In an edge-coloured graph $G$, a  sequence of labelled paths $\cP_1,\dots,\cP_t$ is \emph{$(k,q)$-nice} if 
\begin{itemize}
    \item $\cP_{i}$ has $k$ edges for all $i\in [t]$,
    \item each colour appears at most once on the union of the paths,
    \item for $1\leq i< j\leq t$, the distance between any vertex of $\cP_i$ and any vertex of $\cP_{j}$ in $G$ is greater than $k-1$ (in particular $\cP_i$ and $\cP_{j}$ are vertex-disjoint), and
    \item the number of homomorphic copies of $C_{2k}$ that extend $\cP_{i}$ is at most $q$ for all $i\in [t]$.
\end{itemize}
\end{definition}

\begin{lemma} \label{lem:find paths}
	Let $G$ be a properly edge-coloured  $\mu$-balanced graph with $n$ vertices and average degree $d\geq 2$ and let $s$ be the number of colours in this edge-colouring. Let $2\leq k< d/32$ be an integer such that  $s\leq  \frac{1}{4}nd^{-k+1}\mu^{-k+1}$. Let $\alpha>1$ and assume that $\hom(C_{2k},G)\leq \alpha nd^{k}$. Then, for $t=\lceil \frac{s}{8k\mu}\rceil$, there are at least $(n(d/32)^{k})^t$ sequences of labelled paths $\cP_1,\dots,\cP_t$ which are $(k,2\alpha 16^{k})$-nice.
\end{lemma}

\begin{proof}
	Clearly, it suffices to prove that if $\ell<t$, then any $(k,2\alpha 16^k)$-nice sequence of paths $\cP_1,\dots,\cP_\ell$ extends to at least $n(d/32)^k$ sequences of paths $\cP_1,\dots,\cP_{\ell+1}$ which are $(k,2\alpha 16^k)$-nice.
	
	 Let $G'$ be the subgraph of $G$ obtained from $G$ by removing
	\begin{enumerate}
		\item all vertices within distance at most $k-1$ to some vertex of $\cP_i$ for some $i\leq \ell$
		\item and all edges whose colour appears on $\cP_i$ for some $i\leq \ell$.
	\end{enumerate}

	The first kind of removal deletes at most $$\ell(k+1)(1+\Delta(G)+\dots+\Delta(G)^{k-1})\leq 2(t-1)(k+1)(d\mu)^{k-1}$$ vertices and each such vertex is incident to at most $d\mu$ edges, so the first kind of removal deletes at most $$2(t-1)(k+1)(d\mu)^{k}\leq \frac{s}{4k\mu}(k+1)(d\mu)^{k}<\frac{nd}{8}= \frac{e(G)}{4}$$
	edges. 
	
	Since each colour appears at most $nd\mu/s$ times in the graph, the second kind of removal deletes at most $$\ell k \cdot\frac{nd\mu}{s}\leq (t-1)k\cdot \frac{nd\mu}{s}\leq \frac{nd}{8}= \frac{e(G)}{4}$$ edges.
	
	Hence, $e(G')\geq e(G)/2= nd/4$. By Lemma \ref{lemma:rainbow_paths},  $G'$ contains at least $n(d/16)^k$ labelled rainbow paths of length $k$. Since  $\hom(C_{2k},G)\leq \alpha nd^{k}$, it follows that $G'$ has at most $\alpha nd^{k}\cdot \frac{1}{2\alpha 16^k}=\frac{1}{2}n(d/16)^{k}$ labelled paths of length $k$ which extend to at least $2\alpha 16^{k}$ homomorphic $2k$-cycles. It follows that there are at least $\frac{1}{2}n(d/16)^{k}>n(d/32)^k$ labelled rainbow paths of length $k$ in $G'$ which extend to at most $2\alpha 16^{k}$ homomorphic cycles of length $2k$. Any such path can be chosen to be $\cP_{\ell+1}$, completing the proof.
\end{proof}

\begin{lemma} \label{lem:two paths}
	Let $\cP_1,\dots,\cP_t$ and $\cP'_1,\dots,\cP'_t$ be two sequences of $(k,q)$-nice paths such that the set of endpoints of the two sequences of paths are the same. Assume that some vertex $v$ is an internal vertex of both $\cP_i$ and $\cP'_j$. Then the endpoints of $\cP_i$ and $\cP'_j$ are the same.
\end{lemma}

\begin{proof}
	Suppose this is not the case, then there exists some $i'\neq i$ such that $\cP_{i'}$ shares an endpoint with $\cP'_j$. Since $v$ is an internal vertex of $\cP'_j$, $v$ has distance at most $k-1$ from  an endpoint of $\cP_{i'}$, contradicting that $\cP_1,\dots,\cP_t$ is $(k,q)$-nice.
\end{proof}

\begin{lemma} \label{lem:2-reg with paths}
	Let $G$ be a  properly edge-coloured $\mu$-balanced graph with  $n$ vertices and average degree $d$. Let $\alpha>1$, let $s$ be the number of colours, and suppose that for some integer $2\leq k<d/32$, we have $10^{6k}\alpha \mu^{k+1} nd^{-k}\leq s\leq \frac{1}{4}nd^{-k+1}\mu^{-k+1}$. Assume that $\hom(C_{2k},G)\leq \alpha nd^{k}$. Then $G$ has a 2-regular subgraph in which every colour appears exactly twice or zero times.
\end{lemma}

\begin{proof}
	Let $t=\lceil \frac{s}{8k\mu}\rceil$ and let $\beta=2\alpha 16^{k}$. Then by Lemma \ref{lem:find paths}, there are at least $(n(d/32)^{k})^t$ sequences of labelled paths $\cP_1,\dots,\cP_t$ which are $(k,\beta)$-nice. Now define a collection $\mathcal{R}$ of labelled paths of length $k$ as follows. For any two vertices $u,v$ between which there exists a labelled path of length $k$, choose one uniformly at random which will belong to $\mathcal{R}$ (with probability $1/2$ it starts at $u$ and ends at $v$, and with probability $1/2$ it starts at $v$ and ends at $u$). Now for any $(k,\beta)$-nice sequence of labelled paths $\cP_1,\dots,\cP_t$, the probability that $\cP_i$ belongs to $\mathcal{R}$ for every $1\leq i\leq t$ is at least $(2\beta)^{-t}$ since the number of labelled paths of length $k$ between the endpoints of any $\cP_i$ is at most $2\beta$. Thus, there is a choice for $\mathcal{R}$ for which there are at least $(\frac{n(d/32)^{k}}{2\beta})^t$ sequences of paths $\cP_1,\dots,\cP_t$ which are $(k,\beta)$-nice and satisfy that $\cP_i$ belongs to $\mathcal{R}$ for all $i$.
	
	We claim that $(\frac{n(d/32)^{k}}{2\beta})^t>t!\binom{n}{2t}\binom{s}{kt}$. Indeed,
	$t!\binom{n}{2t}\binom{s}{kt}\leq t^t(\frac{en}{2t})^{2t}(\frac{es}{kt})^{kt}$, so it suffices to prove that $\frac{n(d/32)^{k}}{2\beta}>t(\frac{en}{2t})^2(\frac{es}{kt})^k$. But since $t\geq s/(8k\mu)$ and $s\geq 10^{6k}\alpha \mu^{k+1} nd^{-k}$, we have
	$$t\left(\frac{en}{2t}\right)^2\left(\frac{es}{kt}\right)^k\leq \frac{s}{8k\mu}\left(\frac{4ek\mu n}{s}\right)^2(8e\mu)^{k}\leq \frac{16^{k} e^{k+2} k \mu^{k+1} n^2}{s}<n\left(\frac{d}{32}\right)^{k}\frac{1}{2\beta},$$ so indeed $(\frac{n(d/32)^{k}}{2\beta})^t>t!\binom{n}{2t}\binom{s}{kt}$.
	
	It follows that there exist a set $S$ of $kt$ colours and a set $T$ of $2t$ vertices such that there are more than $t!$ sequences of paths $\mathcal{P}_1,\dots,\mathcal{P}_t$ which are $(k,\beta)$-nice and satisfy that $\cP_i$ belongs to $\mathcal{R}$ for all $i$, the colour set of $\cP_1\cup \dots \cup \cP_t$ is $S$ and the set of endpoints in $\cP_1,\dots,\cP_t$ is $T$. Then clearly there are two such sequences which are not the same set of paths up to reordering, call them $\cP_1,\dots,\cP_t$ and $\cP'_1,\dots,\cP'_t$. By Lemma \ref{lem:two paths}, for any $i,j$, if $\cP_i$ and $\cP'_j$ share an internal vertex, then they have the same endpoints. But since they both belong to $\mathcal{R}$, they must in fact be the same. Hence, the symmetric difference of $\cP_1,\dots,\cP_t$ and $\cP'_1,\dots,\cP'_t$ is a $2$-regular graph on which every colour appears exactly twice or zero times. 
\end{proof}

\begin{lemma} \label{lem:matchings}
	Let $G$ be a  properly edge-coloured $\mu$-balanced graph on $n$ vertices with average degree $d$. Assume that there are $10^5 \mu^2 nd^{-1}\leq s \leq n$ colours. Then $G$ has a 2-regular subgraph in which every colour appears exactly twice or zero times.
\end{lemma}

\begin{proof}
    Let $t=\lceil \frac{s}{16\mu}\rceil$. Note that the number of rainbow matchings in $G$ of size $t$ is at least $(nd/4)^t/t!$. Indeed, having chosen a rainbow matching of size $\ell<t$, the number of ways to choose the next edge is at least $nd/4$ because the number of forbidden edges is at most $$2\ell \cdot  d\mu+\ell \cdot \frac{nd\mu}{s}\leq 2(t-1)d\mu+\frac{(t-1)nd\mu}{s}\leq \frac{nd}{8}+\frac{nd}{16}\leq \frac{nd}{4}.$$ We divide by $t!$ because the order of the edges here does not matter.
    
    We claim that $(nd/4)^t/t!>\binom{n}{2t}\binom{s}{t}$. Indeed, $t!\binom{n}{2t}\binom{s}{t}\leq (t(\frac{en}{2t})^2\frac{es}{t})^t$, so it suffices to prove that $nd/4\geq es(\frac{en}{2t})^2$. But $t\geq \frac{s}{16\mu}$, so $$es\left(\frac{en}{2t}\right)^2\leq es \left(\frac{8en\mu}{s}\right)^2=\frac{64e^{3} n^{2}\mu^2}{s}\leq \frac{nd}{4},$$ where the last inequality holds by our lower bound on $s$.
    
    Hence, there are two rainbow matchings of size $t$ which span the same vertex set and the same colour set. Their symmetric difference is a suitable subgraph.
\end{proof}

\subsection{Combining the above}

\begin{lemma} \label{lem:combine}
	There is an absolute constant $C_0$ such that the following holds. Let $\mu\geq \log n$ and $d\geq \exp(C_0\sqrt{(\log n)(\log \mu)})$. Let $G$ be a properly edge-coloured  $\mu$-balanced graph with $n$ vertices and average degree $d$. Then $G$ has a 2-regular subgraph in which every colour appears exactly twice or zero times.
\end{lemma}

\begin{proof}
	Let $s$ be the number of colours used. Let $k$ be the smallest positive integer such that $s\geq nd^{-k+1/2}$. Then $k\leq \frac{\log n}{\log d}+3/2<\frac{3\log n}{\log d} \leq \frac{3}{C_0}\sqrt{\frac{\log n}{\log \mu}}$. Hence, if $C_0$ is a sufficiently large absolute constant, then $d\geq \mu^{300k}$, so using $\mu\geq \log n\geq k$ and $\mu\geq 10$, we obtain
	\begin{equation}
	    d\geq 10^{100k} k^{100k}\mu^{100k} \label{eqn:d large}.
	\end{equation}
	Fix an absolute constant $C_0$ for which this holds.
	
	If $k=1$, then $s\geq nd^{-1/2}$. But then, noting that the inequality  $d^{1/2}\geq 10^5\mu^2$ holds, we are done by Lemma \ref{lem:matchings}. 
	
	Therefore, we may assume that $k\geq 2$. By the minimality of $k$, we have $s<nd^{-k+3/2}$. Choose $d'$ such that $s=n(d')^{-k+1/2}$. Then $d'\leq d$ since $s\geq nd^{-k+1/2}$. As $n(d')^{-k+1/2}=s<nd^{-k+3/2}$, we have $d'\geq d^{(k-3/2)/(k-1/2)}\geq d^{1/3}$. Let $G'$ be a random subgraph of $G$ where each edge is kept independently with probability $d'/d$. As $\Delta(G)\leq d\mu $, we have $\mathbb{E}(d_{G'}(v))\leq d'\mu$  for every $v\in V(G)$. Hence, by Lemma \ref{lemma:chernoff}, we have
	$$\mathbb{P}(d_{G'}(w)\geq 2d'\mu)\leq e^{-d'\mu/3}\leq \frac{1}{n^2}.$$
	Similarly, if $c$ is a colour, then the probability that $c$ is used more than $2nd'\mu/s$ times in $G'$ is at most $1/n^{2}$. Finally, $\mathbb{E}(e(G'))=nd'/2$, so by Lemma \ref{lemma:chernoff} again, we have $\mathbb{P}(e(G')\leq nd'/4)\leq e^{-nd'/16}<1/4$ and $\mathbb{P}(e(G')\geq nd')\leq e^{-nd'/6}<1/4$. Therefore, with positive probability, we have $nd'/4\leq e(G')\leq nd'$, $\Delta(G')\leq 2d'\mu$, and every colour is used at most $2nd'\mu/s$ times in $G'$. This means that there exists a choice for $G'$ such that $G'$ is $4\mu$-balanced with average degree between $d'/2$ and $2d'$; fix such a choice.

	Let $\alpha_0=(128k^{3/2}(4\mu)^{1/2})^{2k}$ and $\alpha=\alpha_0 4^{k}$.
	
	Suppose that every subgraph $G''$ of $G'$ with at least $e(G')/2$ edges satisfies $\hom(C_{2k},G'')\geq \alpha_0 n (2d')^{k}$. Note that $s=n(d')^{-k+1/2}$ implies that $s\leq n(2d')^{-k+1}(100(4\mu) k)^{-4k}$, using $d'\geq d^{1/3}$ and (\ref{eqn:d large}). Then, since $G'$ has average degree at most $2d'$, by Lemma \ref{lem:2-reg with cycles}, $G'$ has a $2$-regular subgraph in which every colour appears exactly twice or zero times.
	
	Otherwise, $G'$ has a subgraph $G''$ with at least $e(G')/2$ edges such that $\hom(C_{2k},G'')\leq \alpha_0 n (2d')^{k}$. We may assume that $G''$ is a spanning subgraph of $G'$ (by adding back some isolated vertices). If $G''$ has average degree $d''$, then $d''\geq d'/2$, so $G''$ is $8\mu$-balanced and $\hom(C_{2k},G'')\leq \alpha n(d'')^{k}$. Note that $s=n(d')^{-k+1/2}$ and $d'/2\leq d''\leq d'$ imply that $2^{-k+1/2}n(d'')^{-k+1/2}\leq s\leq n(d'')^{-k+1/2}$. Furthermore, using $d''\geq d'/2\geq d^{1/3}/2$ and (\ref{eqn:d large}), we obtain $10^{6k}\alpha (8\mu)^{k+1}n (d'')^{-k}\leq s\leq \frac{1}{4}n(d'')^{-k+1}(8\mu)^{-k+1}$. Hence, we can apply Lemma \ref{lem:2-reg with paths} to conclude that $G''$ has a $2$-regular subgraph in which every colour appears exactly twice or zero times. 
\end{proof}

It is not hard to see that, using Lemma \ref{lem:balanced regularization}, Lemma \ref{lem:combine} implies Theorem~\ref{thm:2-reg}.

\begin{proof}[Proof of Theorem \ref{thm:2-reg}] 
Let $C_0$ be the constant given by Lemma \ref{lem:combine}, let $C=10C_0$  and assume that $n$ is sufficiently large with respect to $C$. Let $d=\exp(C\sqrt{(\log n) (\log\log n)})$ and let $\mathcal{G}$ be a linear $3$-graph on $n$ vertices with at least $nd$ edges. Let $\lambda=\lceil \log n\rceil$ and let $\mu=96\lambda^6$. By Lemma \ref{lem:balanced regularization}, $\mathcal{G}$ contains a $3$-partite $\mu$-balanced subhypergraph $\mathcal{H}$ with at least $|V(\mathcal{H})|d/(81\lambda^3)$ edges and parts $X,Y,Z$ such that $|X|\leq |Y|=|Z|$. Let $s=|X|$ and $n_0=|Y|+|Z|$. 

Let $G$ be the edge coloured graph on vertex set $Y\cup Z$, in which for $x\in X$, $y\in Y$ and $z\in Z$, $yz$ is an edge of colour $x$ if $xyz\in E(\mathcal{H})$. Then $G$  is a $\mu$-balanced properly edge coloured graph on $n_0$ vertices with average degree $$d_0=\frac{2e(G)}{n_0}=\frac{2e(\mathcal{H})}{n_0}\geq \frac{d}{81\lambda^{3}}> \exp(C_0\sqrt{(\log n_0) (\log\mu)}).$$  Hence, by Lemma \ref{lem:combine}, $G$ contains a 2-regular subgraph in which every colour appears exactly twice or zero times. This subgraph corresponds to a 2-regular subhypergraph of $\mathcal{H}$, finishing the proof.
\end{proof}

\section{$2$-regular subhypergraphs of bounded size}\label{sect:2-reg_bounded}

In this section, we prove the following theorem, which (as we will see in a moment) implies Theorem \ref{thm:bounded 2-reg}.

\begin{theorem} \label{thm:two rainbow cycles}
    For any integer $\ell\geq 2$, there is a constant $C=C(\ell)$ such that the following holds. Let $G$ be a graph with at most $n$ vertices and at least $Cn^{3/2+1/\ell}$ edges, with a proper edge-colouring that uses at most $n$ colours. Then $G$ has two vertex-disjoint rainbow cycles of length $2\ell$ with the same set of colours used on both cycles.
\end{theorem}

In order to see that Theorem \ref{thm:two rainbow cycles} implies Theorem \ref{thm:bounded 2-reg}, let $G$ be an $n$-vertex linear $3$-graph with at least $\frac{9}{2}Cn^{3/2+1/\ell}$ edges for the constant $C$ provided by Theorem \ref{thm:two rainbow cycles}. By standard probabilistic arguments, $G$ has a $3$-partite subhypergraph $G'$ with at least $Cn^{3/2+1/\ell}$ edges. By treating vertices in one of the parts as colours, we obtain a properly edge-coloured graph with at most $n$ vertices and at least $Cn^{3/2+1/\ell}$ edges, in which at most $n$ colours are used. By Theorem~\ref{thm:two rainbow cycles}, there exist two vertex-disjoint rainbow cycles of length $2\ell$ using the same set of colours. The union of these two cycles corresponds to a $2$-regular subhypergraph in $G'$ which has precisely $4\ell$ edges.

The proof of Theorem \ref{thm:two rainbow cycles} is similar to that of Theorem 1.11 from \cite{Jan20}. The key ingredient is Theorem 3.1 from \cite{Jan20}, which we now recall.

\begin{lemma}[\cite{Jan20}] \label{lem:good cycle}
	Let $\ell\geq 2$ and $s$ be positive integers. Then there exists a constant $c=c(\ell,s)$ with the following property. Suppose that $G=(V,E)$ is a graph with $n$ vertices and at least $cn^{1+1/\ell}$ edges. Let $\sim$ be a symmetric binary relation on $V$ such that for every $u\in V$ and $v\in V$, $v$ has at most $s$ neighbours $w\in V$ which satisfy $u\sim w$. Let $\approx$ be a binary relation on $E$ such that for every $uv\in E$ and $w\in V$, $w$ has at most $s$ neighbours $z\in V$ which satisfy $uv\approx wz$. Then $G$ contains a $2\ell$-cycle $x_1x_2\dots x_{2\ell}$ such that $x_i\not \sim x_j$ for every $i\neq j$ and $x_ix_{i+1}\not \approx x_jx_{j+1}$ for every $i\neq j$ (where $x_{2\ell+1}:=x_1$).
\end{lemma}

\begin{proof}[Proof of Theorem \ref{thm:two rainbow cycles}]
Let $C=C(\ell)$ be sufficiently large and let $G$ be a graph with at most $n$ vertices and at least $Cn^{3/2+1/\ell}$ edges, with a proper edge-colouring that uses at most $n$ colours. Define a graph $\cG$ as follows. The vertex set of $\cG$ is the set of ordered pairs of distinct vertices in $G$. Vertices $(u_1,u_2)\in V(\cG)$ and $(v_1,v_2)\in V(\cG)$ are joined by an edge in $\cG$ if $u_1v_1$ and $u_2v_2$ are edges of the same colour in $G$. In this case, we colour the edge in $\cG$ between $(u_1,u_2)$ and $(v_1,v_2)$ by this colour. Assuming that $G$ has $t$ red edges, the number of pairs of red edges in $G$ is $\binom{t}{2}$ and hence $\cG$ has at least $\binom{t}{2}$ red edges. As the number of colours is at most $n$, this implies by convexity that $\cG$ has at least $n\binom{Cn^{1/2+1/\ell}}{2}\geq \frac{C^2}{4}n^{2+2/\ell}$ edges. Since $|V(\cG)|\leq |V(G)|^2\leq n^2$, it follows that $e(\cG)\geq \frac{C^2}{4}|V(\cG)|^{1+1/\ell}$.

For $(u_1,u_2),(v_1,v_2)\in V(\cG)$, write $(u_1,u_2)\sim (v_1,v_2)$ if $\{u_1,u_2\}\cap \{v_1,v_2\}\neq \emptyset$. Moreover, for $e,f\in E(\cG)$, write $e\approx f$ if $e$ and $f$ have the same colour. Note that for any $(u_1,u_2),(v_1,v_2)\in V(\cG)$, the number of $(w_1,w_2)\in V(\cG)$ with $(v_1,v_2)(w_1,w_2)\in E(\cG)$ and $(u_1,u_2)\sim (w_1,w_2)$ is at most $4$. Indeed, we must have $w_i=u_j$ for some $i,j\in \{1,2\}$ and each case leaves at most one possibility for $(w_1,w_2)$. For example, when $w_1=u_2$, then $v_2w_2$ must have the same colour as $v_1w_1=v_1u_2$, and since the colouring of $G$ is proper, there is at most one such $w_2$. The other three cases are nearly identical. Similarly, for any $e\in E(\cG)$ and $(w_1,w_2)\in V(\cG)$, there is at most one $(z_1,z_2)\in V(\cG)$ such that $(w_1,w_2)(z_1,z_2)\in E(\cG)$ and $e\approx (w_1,w_2)(z_1,z_2)$ since both $w_1z_1$ and $w_2z_2$ must have the same colour as $e$. This means that if $\frac{C^2}{4}\geq c(\ell,4)$, then we can use Lemma \ref{lem:good cycle} with $s=4$ to find a $2\ell$-cycle $x_1x_2\dots x_{2\ell}$ in $\cG$ such that $x_i\not \sim x_j$ for every $i\neq j$ and $x_ix_{i+1}\not \approx x_jx_{j+1}$ for every $i\neq j$. This corresponds to two vertex-disjoint rainbow cycles of length $2\ell$ in $G$ which use the same set of colours.
\end{proof}

\section{$r$-regular subhypergraphs}\label{sect:r-reg}

In this section, we prove Theorem \ref{thm:k-uni_r-reg}. We start with the lower bound, which we restate as follows.

\begin{theorem}
For every positive integer $k\geq 2$ there exists $C_0=C_0(k)>0$ such that the following holds. Let $2\leq r\leq n$ be positive integers. Then there exists a linear $k$-graph $G$ on $n$ vertices with at least $$C_0 r^{\frac{1}{k-1}} n^{2-\frac{r}{(r-1)(k-1)}}$$ edges containing no $r$-regular subhypergraph.
\end{theorem}

\begin{proof}
If $\frac{r}{(r-1)(k-1)}\geq 1$, then the statement is trivial as it is easy to find a linear $k$-graph on $n$ vertices with $\Theta_k(rn)$ edges and maximum degree less than $r$. Suppose therefore that $\frac{r}{(r-1)(k-1)}< 1$. In particular, we have $k\geq 3$. Assume that $n$ is sufficiently large with respect to $k$. First, we define a $k$-graph $G_0$ as follows. Let the vertex set of $G_0$ be $A\cup B$ such that $|A|=c_0 r^{\frac{r}{(r-1)(k-1)}} n^{1-\frac{r}{(r-1)(k-1)}}$ and $|B|=n-|A|$, where $c_0=c_0(k)>0$ is some sufficiently small number depending only on $k$. Let $p=\frac{1}{8(k-1)!n^{k-2}}$ and let each $k$-element subset of $A\cup B$ with exactly one vertex in $A$ be an edge of $G_0$ with probability $p$, independently of each other. Let $H_0$ denote the $(k-1)$-graph we get by restricting the edges of $G_0$ to $B$ (we include edges without multiplicities).

Let $X$ be the number of edges of $G_0$. Then $$\mathbb{E}(X)=|A|\binom{|B|}{k-1}p\geq \frac{|A|n^{k-1}p}{2(k-1)!}.$$

Let $Y$ be the number of pairs of edges of $G_0$ which intersect in more than 2 vertices. Then 
$$\mathbb{E}(Y)\leq n^{2k-4}|A|^{2}p^2+|A|n^{2k-3}p^2\leq 2|A|n^{2k-3}p^2\leq \frac{\mathbb{E}(X)}{2}.$$
Indeed, $n^{2k-4}|A|^{2}$ is an upper bound on the number of pairs of edges that intersect in at least 2 elements, none of which is in $A$. On the other hand,  $|A|n^{2k-3}$ is an upper bound on the number of pairs of edges that intersect in at least 2 elements, one of which is in $A$.

Say that a non-empty subhypergraph $H'$ of $H_0$ is \emph{bad} if $v(H')\leq |A|(k-1)$ and $e(H')=\frac{v(H')r}{k-1}$. Let $Z$ denote the number of bad subhypergraphs of $H_0$. Write $q=1-(1-p)^{|A|}<p|A|$, then $q$ is the probability that a $(k-1)$-element subset of $B$ is an edge of $H_0$. We have
\begin{align*}
    \mathbb{E}(Z)&\leq \sum_{m=k-1}^{|A|(k-1)}\binom{|B|}{m}\binom{m^{k-1}}{mr/(k-1)}q^{\frac{mr}{k-1}}\leq \sum_{m=k-1}^{|A|(k-1)}\left(\frac{en}{m}\right)^{m}\left(\frac{em^{k-2}k}{r}\right)^{\frac{mr}{k-1}}q^{\frac{mr}{k-1}}\\
    &=\sum_{m=k-1}^{|A|(k-1)}\left(\frac{en\cdot (em^{k-2}kq)^{\frac{r}{k-1}}}{mr^{\frac{r}{k-1}}}\right)^{m}.
\end{align*}
Using $m\leq |A|(k-1)$ and $q\leq |A|p$, we have 
\begin{equation}
    \frac{en(e m^{k-2} k q)^{\frac{r}{k-1}}}{mr^{\frac{r}{k-1}}}< n |A|^{r-1} \left(\frac{c_1p}{r}\right)^{\frac{r}{k-1}}, \label{eqn:third}
\end{equation}
where $c_1=c_1(k)>0$ is some number depending only on $k$ (for the inequality above, note that the left hand side is non-decreasing in $m$ since the exponent of $m$ is $\frac{k-2}{k-1}r-1\geq r/2-1\geq 0$, where we used our earlier assumption that $k\geq 3$). Assuming $c_0$ is sufficiently small with respect to $c_1$, the right hand side of (\ref{eqn:third}) is at most $1/2$. Hence, $\mathbb{E}(Z)\leq 1$.

For each pair of edges in $G_0$ sharing at least 2 vertices, delete one of the edges, and for each bad subhypergraph $H'$ of $H_0$, delete an edge of $G_0$ whose restriction to $B$ is contained in $H'$. Let the resulting subhypergraph be $G$. Let $H$ be the $(k-1)$-graph we get by restricting the edges of $G$ to $B$.

Note that $G$ is a linear hypergraph, and no subhypergraph of $H$ on $m\leq |A|(k-1)$ vertices spans exactly $\frac{mr}{k-1}$ edges. The latter implies that $G$ contains no $r$-regular subhypergraph. Indeed, otherwise, let $U\subset V(G)$ be the vertex set of an $r$-regular subhypergraph of $G$. Then we have $|U\cap B|=(k-1)|U\cap A|$, as the number of edges in such a hypergraph is $r|U\cap A|=\frac{r|U\cap B|}{k-1}$. But then $|U\cap B|\leq (k-1)|A|$, which implies that $e(H[U\cap B])\neq\frac{r |U\cap B|}{k-1}$, a contradiction.

Finally, $e(G)\geq X-Y-Z$, so 
$$\mathbb{E}(e(G))\geq \mathbb{E}(X-Y-Z)\geq \frac{\mathbb{E}(X)}{2}-1\geq C_0 r^{\frac{r}{(r-1)(k-1)}} n^{2-\frac{r}{(r-1)(k-1)}},$$
where $C_0=C_0(k)>0$ is some number depending only on $k$. Hence, there exists a choice for $G_0$ such that $e(G)\geq C_0 r^{\frac{r}{(r-1)(k-1)}} n^{2-\frac{r}{(r-1)(k-1)}}>C_0r^{\frac{1}{k-1}}n^{2-\frac{r}{(r-1)(k-1)}}$.
\end{proof}

 Now let us prove the upper bound in Theorem \ref{thm:k-uni_r-reg}. A key ingredient in our proof is the recent breakthrough on the celebrated Sunflower conjecture \cite{ER60}. Given a positive integer $r$, a collection of $r$ sets is an \emph{$r$-sunflower} if there exists a (possibly empty) set $X$ (called the \emph{core}) such that the intersection of any two distinct elements in the collection is  $X$. It was recently proved by Alweiss, Lovett, Wu, and Zhang \cite{ALWZ20} that every $t$-uniform family of size at least $(\log t)^t\cdot (r\log\log t)^{O(t)}$ contains an $r$-sunflower, which was subsequently improved by Frankston, Kahn, Narayanan, and Park \cite{FKNP21} and Rao \cite{Rao21}. We use the following version, proved by Rao \cite{Rao21}.

\begin{lemma}\label{lemma:sunflower}
There is a constant $\alpha>1$ such that any family with  more than $(\alpha r\log (tr))^t$ sets of size $t$ contains an $r$-sunflower. 
\end{lemma}

In what follows, we first consider the special case when $G$ is $\mu$-balanced (see~(\ref{def1})). 

\begin{lemma}\label{lemma:num_matching}
Let $G$ be a $\mu$-balanced, $k$-partite linear  $k$-graph on $n$ vertices with vertex classes $X_1,\dots,X_k$ such that $|X_1|\leq\dots \leq |X_k|$. Let $t\leq \frac{|X_1|}{2k\mu}+1$ be a positive integer. Then $G$ contains at least $(\frac{e(G)}{2t})^t$ matchings of size $t$. 
\end{lemma}

\begin{proof}

We generate a matching of size $t$ with the following greedy procedure. Let $G_1=G$. In step $i$ ($1\leq i\leq t$), if $G_i$ is already defined and $G_i$ is nonempty, select an arbitrary edge $e_i$ of $G_i$, and let $G_{i+1}$ be the hypergraph we get after removing $e_i$ and all edges intersecting $e_i$ from $G_{i}$. Clearly, $\{e_1,\dots,e_{i}\}$ is a matching. Note that, since $\max_{w\in X_i} d_G(w)\leq \frac{\mu e(G)}{|X_i|}$, every edge of $G$ intersects at most $\sum_{i=1}^{k}\frac{\mu e(G)}{|X_i|}\leq \frac{k\mu e(G)}{|X_1|}$ edges, so $$e(G_i)\geq e(G)-\frac{(i-1)\cdot k \cdot \mu \cdot e(G)}{|X_1|}\geq e(G)-\frac{(t-1)\cdot k \cdot \mu \cdot e(G)}{|X_1|}\geq \frac{e(G)}{2}.$$

Hence, at step $i$, there are at least $\frac{1}{2}e(G)$ choices for the edge $e_i$, so in total there are at least $(\frac{1}{2}e(G))^t$ sequences $e_1,\dots,e_t$ such that $\{e_1,\dots,e_{t}\}$ is a matching. But then there are at least $\frac{1}{t!}(\frac{1}{2}e(G))^t\geq (\frac{e(G)}{2t})^{t}$ matchings in $G$, as each matching corresponds to at most $t!$ such sequences.
\end{proof}

\begin{lemma}\label{lemma:r-reg_balanced}
 There exists a constant $c_1=c_1(k)>0$ such that the following holds. Let $G$ be a $\mu$-balanced, $k$-partite linear $k$-graph on $n$ vertices such that $$e(G)>c_1r^{\frac{1}{k-1}}\mu n^{2-\frac{1}{k-1}}(\log n)^{\frac{1}{k-1}}.$$ Then $G$ contains an $r$-regular subhypergraph.
\end{lemma}

\begin{proof}
We show that $c_1=8k\alpha^{\frac{1}{k-1}}$ suffices, where $\alpha$ is the constant guaranteed by Lemma~\ref{lemma:sunflower}. Let $X_1,\dots,X_k$ be the vertex classes of $G$, where $|X_1|\leq\dots\leq |X_k|$. Let $t=\left\lceil \frac{|X_1|}{2k\mu} \right\rceil$. By Lemma~\ref{lemma:num_matching}, $G$ contains at least $N=(\frac{e(G)}{2t})^t$ matchings of size $t$. Each such matching covers exactly $t$ vertices in $X_i$ for every $i\in [t]$, so by the Pigeonhole principle, there exists a set $U\subset V(G)$ of size $tk$ such that at least 
$$\frac{N}{\prod_{i=1}^{k}\binom{|X_i|}{t}}\geq \frac{N}{\binom{|X_1|}{t}\binom{n}{t}^{k-1}}$$
matchings of size $t$ span $U$. Here, using the general inequality $\binom{a}{b}\leq (ea/b)^{b}$, we can further bound the right hand side from below by
$$N\left(\frac{t^k}{e^k|X_1| n^{k-1}}\right)^t=\left(\frac{e(G) t^{k-1}}{2 e^k |X_1|n^{k-1}}\right)^t\geq \left(\frac{e(G) |X_1|^{k-2}}{(2e)^k(k\mu)^{k-1} n^{k-1}}\right)^{t}.$$
Note that $|X_1|\geq e(G)/n$, as the maximum degree of $G$ is at most $n$. Hence, by choosing a sufficiently large constant $c_1$, we can  bound the right hand side from below by
$$\left(\frac{e(G)^{k-1}}{(2e)^{k}(k\mu)^{k-1} n^{2k-3}}\right)^t>(2\alpha r\log n)^{t}.$$
Let $\mathcal{F}$ be the family of sets whose ground set consists of all the edges of $G$ in $U$ and whose sets are the matchings spanning $U$. As $n^2\geq tr$, Lemma \ref{lemma:sunflower} guarantees that $\mathcal{F}$ contains an $r$-sunflower $S$. This is a collection of $r$
matchings spanning $U$, such that every edge either belongs to all the matchings or only to one of them. Remove the edges contained in the core of $S$, let the resulting set be $S'$. Then the edges in $S'$ form an $r$-regular subhypergraph of $G$, finishing the proof.
\end{proof}

Now we are ready to prove the upper bound of Theorem \ref{thm:k-uni_r-reg}. In fact, we prove the following more precise formulation.

\begin{theorem}
For every positive integer $k\geq 2$, there exists a constant $c=c(k)>0$ such that the following holds. Let $r$ be a positive integer, and let $G$ be a linear $k$-graph on $n$ vertices with at least $$cr^{\frac{1}{k-1}}n^{2-\frac{1}{k-1}}(\log n)^{2k+\frac{1}{k-1}}$$ edges. Then $G$ contains an $r$-regular subhypergraph.
\end{theorem}

\begin{proof}
We show that $c=10^{k}c_1k^{k}/k!$ suffices, where $c_1$ is the constant given by Lemma \ref{lemma:r-reg_balanced}.  Let $\lambda=\lceil \log n\rceil<2\log n$. By Lemma \ref{lem:DHL regularization}, $G$ contains a $2\lambda^{k}$-balanced subhypergraph $H$ with at least $e(G)k!/(k\lambda)^{k}$ edges. Then, writing $\mu=2\lambda^{k}$, we have $$e(H)\geq c_1r^{\frac{1}{k-1}}\mu n^{2-\frac{1}{k-1}}(\log n)^{\frac{1}{k-1}},$$ so we can apply Lemma \ref{lemma:r-reg_balanced} to conclude that $H$ contains an $r$-regular subhypergraph.
\end{proof}

\section{Hypergraph immersion} \label{sect:immersion}

In this short section, we prove Theorem \ref{thm:immersion}.

\begin{proof}[Proof of Theorem \ref{thm:immersion}]
Let $C$ be the constant defined by Theorem \ref{thm:2-reg}, and let $G$ be a 3-graph with at least $n^2\exp(3C\sqrt{(\log n)(\log\log n)})$ edges. We show that $G$ contains a 0-immersion of a surface.

Define the auxiliary 3-graph $H$ as follows. Let the vertex set of $H$ be the set of unordered pairs of vertices $uv$ with $u,v\in V(G)$, and for every $uvw\in E(G)$, let the set $\{uv,vw,uw\}$ be an edge of $H$. Then $H$ is a linear 3-graph with $N=\binom{n}{2}$ vertices and $e(G)\geq N\exp(C\sqrt{(\log N)(\log\log N)})$ edges. Therefore, $H$ contains a 2-regular subhypergraph $H'$. This 3-graph $H'$ corresponds to a subhypergraph $G'$ of $G$ in which any pair of vertices is contained in exactly 2 or 0 edges. In other words, the link graph of any vertex in $G$ is a 2-regular graph, that is, a vertex-disjoint union of cycles.

Define the sequence of 3-graphs $S_0,S_1,\dots$ as follows. Let $S_0$ be the 3-graph we get after removing the isolated vertices from $G'$. Suppose that $S_i$ has a vertex $v\in V(S_i)$ whose link graph is the union of $k$ cycles $C_1,\dots,C_k$, where $k\geq 2$. Remove $v$ and add $k$ new vertices $v_1,\dots,v_k$, called the \emph{clones of $v$} such that the link graph of $v_i$ is $C_i$. Let the resulting 3-graph be $S_{i+1}$.  If there is no such vertex $v$, stop, and set $S=S_i$. 

Observe that $S_{i+1}$ has the same number of $1$- and $2$-faces as $S_i$. Furthermore, the link graph of any vertex other than the clones of $v$ remains the same (up to isomorphism). This implies that a clone will not get cloned again, so the sequence ends in a finite number of steps. In the end, the link graph of every vertex of $S$  is a cycle, therefore $S$ is homeomorphic to the disjoint union of closed surfaces. Furthermore there is a homomorphism $\phi:S\rightarrow S_0$ which is injective on the $1$- and $2$-faces. Indeed, let $\phi$ be the homomorphism induced by the map $\phi_0:V(S)\rightarrow V(S_0)$ that maps every clone to its original, and if a vertex was not cloned, then to itself. Hence, $G$ contains a 0-immersion of a surface, finishing the proof.
\end{proof}

\section{Concluding remarks}

In this paper we proved an asymptotically tight upper bound for the maximum possible number of edges in a $3$-uniform linear hypergraph without a $2$-regular subhypergraph. We also improved the best known bounds for the more general problem of maximizing the number of edges in $k$-uniform linear hypergraphs without an $r$-regular subhypergraph (see Theorem \ref{thm:k-uni_r-reg}). We expect that when $r\geq 3$, proving that the lower bound is tight would be difficult, even in the $3$-uniform case. The problem is closely related to the following question, recently considered by the authors \cite{JST22}. Given an $n$-vertex graph $G$ with average degree $d$ and a positive integer $r$, what is the largest possible order of the smallest $r$-regular subgraph of $G$? The authors conjectured that for some $C=C(r)$, if $C\log \log n\leq d\leq n^{\frac{r-2}{r}}$, then $G$ contains an $r$-regular subgraph on roughly at most $nd^{-\frac{r}{r-2}}$ vertices, which would be tight by a simple random construction. Note that for $k=3$, the lower bound in our Theorem \ref{thm:k-uni_r-reg} is $f_r^{(3)}(n)\gtrsim n^{2-\frac{r}{2r-2}}$. We claim that the tightness of this bound would imply the above conjecture of the authors about small $r$-regular subgraphs in graphs, for $d\approx n^{\frac{r-2}{2r-2}}$. Indeed, assuming that the conjecture fails to hold for this value of $d$, there is an $n$-vertex graph $G$ with average degree $d$ in which every $r$-regular subgraph has at least $\omega(nd^{-\frac{r}{r-2}})=\omega(n^{\frac{r-2}{2r-2}})$ vertices. Since $G$ has $n$ vertices and $\Theta(n^{1+\frac{r-2}{2r-2}})$ edges, it follows by a classical regularization lemma of Erd\H os and Simonovits~\cite{ES70} that $G$ has a subgraph $G'$ with $m$ vertices, $\Theta(m^{1+\frac{r-2}{2r-2}})$ edges and maximum degree $O(m^{\frac{r-2}{2r-2}})$ for some $m=\omega(1)$. By the maximum degree condition, $G'$ has a proper edge-colouring with $O(m^{\frac{r-2}{2r-2}})$ colours. By defining a new vertex $u_c$ for each colour $c$ and taking a 3-edge $u_c vw$ for every edge $vw$ of colour $c$ in $G'$, we obtain a linear $3$-uniform hypergraph $\cG$ with roughly $m$ vertices and $\Theta(m^{1+\frac{r-2}{2r-2}})=\Theta(m^{2-\frac{r}{2r-2}})$ edges. By the assumption that $f_r^{(3)}(m)\approx m^{2-\frac{r}{2r-2}}$, this implies that $\cG$ has an $r$-regular subhypergraph. This corresponds to an $r$-regular subgraph in $G'$ in which every colour is used zero times or $r$ times. By assumption, this subgraph has $\omega(n^{\frac{r-2}{2r-2}})$ vertices (since it is an $r$-regular subgraph of $G$). It follows easily that the number of colours used on this subgraph must be $\omega(n^{\frac{r-2}{2r-2}})$, but this is a contradiction as we have only $O(m^{\frac{r-2}{2r-2}})$ colours and $m\leq n$.

\bibliographystyle{abbrv}
\bibliography{bibliography}

\end{document}